\documentclass{article}
\usepackage[height=9in,width=6.6in]{geometry}
\usepackage{graphicx}
\usepackage{amsmath}
\usepackage{amssymb}
\usepackage{amsthm}
\usepackage{bm}
\usepackage{mathrsfs}
\usepackage{amssymb}
\usepackage{natbib}
\usepackage{enumerate}

\newcommand{\half}{\tfrac{1}{2}}
\newcommand{\reals}{\mathbb{R}}
\newcommand{\sml}[1]{{\small #1}}
\newcommand{\fromto}[3]{\sml{$#1 \le #2 \le #3$}}
\newcommand{\Ph}{\hat{P}}
\newcommand{\Ah}{\hat{A}}
\newcommand{\Bh}{\hat{B}}
\newcommand{\ah}{\hat{a}}
\newcommand{\bh}{\hat{b}}
\newcommand{\kron}{\otimes}
\newcommand{\norm}[1]{\|{#1}\|}
\DeclareMathOperator{\rank}{rank}
\DeclareMathOperator{\trace}{tr}
\newcommand{\nlsum}{\sum\nolimits}

\newtheorem{theorem}{Theorem}
\newtheorem{lemma}[theorem]{Lemma}
\newtheorem{prop}[theorem]{Proposition}
\newtheorem{corr}[theorem]{Corollary}
\newtheorem{conj}[theorem]{Conjecture}
\theoremstyle{definition}

\numberwithin{equation}{section}

\begin{document}

\title{Explicit eigenvalues of certain scaled trigonometric matrices}
\author{Suvrit Sra\\ \footnotesize \it Max Planck Institute for Intelligent Systems\\
  \footnotesize\it 72076 T\"ubingen, Germany}
\maketitle

\begin{abstract}
  In a very recent paper ``\emph{On eigenvalues and equivalent transformation of trigonometric matrices}'' (D. Zhang, Z. Lin, and Y. Liu, LAA 436, 71--78 (2012)), the authors motivated and discussed a trigonometric matrix that arises in the design of finite impulse response (FIR) digital filters. The eigenvalues of this matrix shed light on the FIR filter design, so obtaining them in closed form was investigated. Zhang \emph{et al.}\ proved that their matrix had rank-4 and they conjectured closed form expressions for its eigenvalues, leaving a rigorous proof as an open problem. This paper studies trigonometric matrices significantly more general than theirs, deduces their rank, and derives closed-forms for their eigenvalues. As a corollary, it yields a short proof of the conjectures in the aforementioned paper.
\end{abstract}

\section{Introduction}
Matrices generated by trigonometric functions enjoy a wide-range of appealing analytic and numerical properties, whereby they are important across a wide variety of problems. The present paper is motivated by the very recent article (in this journal) of~\citet{zhLiLi12}, who studied certain trigonometric matrices arising in the design finite impulse response (FIR) digital filters. In that paper, the authors motivated the importance of obtaining eigenvalues of these matrices for their FIR design task.

Specifically, \citet{zhLiLi12} considered \emph{trigonometric matrices} of the form
\begin{equation}
  \label{eq.71}
  P(\omega) :=
  \begin{bmatrix}
    A(\omega) & B(\omega)\\
    B(\omega)^T & A(\omega)
  \end{bmatrix},
\end{equation}
where $A(\omega)$ and $B(\omega)$ are $n \times n$ matrices ($n\ge 2$), and $\omega$ is the digital frequency variable with $0 \le \omega \le 2\pi$. The entries of $A(\omega)$ and $B(\omega)$ (denoted with lowercase letters) are given by
\begin{equation}
  \label{eq.18}
    a_{ij}(\omega) := \half(i+j-2)\cos(i\omega - j\omega),\qquad\text{and}\qquad b_{ij}(\omega) := \half(i+j-2)\sin(i\omega - j\omega). 
\end{equation}
For the matrix $P(\omega)$ defined using~\eqref{eq.18}, \citet{zhLiLi12} made the following conjecture.

\begin{conj}[\citep{zhLiLi12}]
  \label{conj.one}
  $P(\omega)$ has one positive eigenvalue $\lambda_+$ and one negative eigenvalue $\lambda_-$, both of which are of multiplicity 2, are independent of $\omega$, and are given by
  \begin{equation}
    \label{eq.72}
      \lambda_+ = \frac{n}{4}\left(n-1 + \sqrt{\frac{4n^2-6n+2}{3}}\right)\quad\quad\lambda_- = \frac{n}{4}\left(n-1 - \sqrt{\frac{4n^2-6n+2}{3}}\right).
  \end{equation}
  The other eigenvalues of $P(\omega)$ are zeros.
\end{conj}

In this paper, we consider trigonometric matrices more general than~\eqref{eq.18}, and derive their eigenvalues in closed form. As a corollary (Corollary~\ref{corr.conj}) we obtain a short proof of Conjecture~\ref{conj.one}. Our more general matrices may be useful in applications such as those mentioned by~\citet{zhLiLi12} and references therein.

\section{Generalized trigonometric matrices}
In this paper we consider \emph{generalized trigonometric matrices}
\begin{equation}
  \label{eq.26}
  P(\omega) :=
  \begin{bmatrix}
    A(\omega) & B(\omega)\\
    B^T(\omega) & A(\omega)
  \end{bmatrix},
\end{equation}
that have the same form as~(\ref{eq.71}) but with more general component matrices $A(\omega)$ and $B(\omega)$ defined by
\begin{equation}
  \label{eq.24}
    a_{ij}(\omega) := l_{ij}\cos(x_i - x_j),\qquad\text{and}\qquad b_{ij}(\omega) := l_{ij}\sin(x_i - x_j),
\end{equation}
where $x_i$ and $x_j$ are components of an arbitrary vector $x \in \reals^n$, and $l_{ij}$ are the entries of an arbitrary, symmetric rank-2 matrix $L$ defined as
\begin{equation}
  \label{eq.25}
  L := lh^T + hl^T,\quad\text{for vectors}\ l, h \in \reals^n.
\end{equation}
In particular, matrix~\eqref{eq.71}  corresponds to $L=le^T+el^T$ for vector $l$ having entries $l_i=(i-1)/2$, where $e$ denotes the all ones vector.

To analyze the rank of $P(\omega)$ and its eigenvalues we need to introduce more notation. First, we drop $\omega$, because the end results do not depend on it; then, we introduce the \emph{pure} trigonometric matrix
\begin{equation}
  \label{eq.27}
  \Ph :=
  \begin{bmatrix}
    \Ah & \Bh\\
    \Bh^T & \Ah
  \end{bmatrix},
\end{equation}
where the entries of the component matrices $\Ah$ and $\Bh$ are given by
\begin{equation}
  \label{eq.29}
    \ah_{ij} := \cos(x_i - x_j),\qquad\text{and}\qquad \bh_{ij} := \sin(x_i - x_j),
\end{equation}
Using this notation we see that the matrix~(\ref{eq.26}) can be written as
\begin{equation}
  \label{eq.30}
  P = \Ph \odot \left(
    \begin{bmatrix}
      1 & 1\\
      1 & 1
    \end{bmatrix} \kron L
  \right) = \Ph \odot
  \begin{bmatrix}
    L & L\\
    L & L
  \end{bmatrix},
\end{equation}
where $\kron$ is the Kronecker and $\odot$ the Hadamard product. But we know that $\rank(X \odot Y) \le \rank(X)\rank(Y)$ for any $X$ and $Y$; moreover $\rank(
\tiny\begin{bmatrix}
  L & L\\ L & L
\end{bmatrix}
) =\rank(L)= 2$, whereby from~\eqref{eq.30} it follows that
\begin{equation}
  \label{eq.31}
  \rank(P) \le 2\rank(\Ph).
\end{equation}
We note that inequality~\eqref{eq.31} may be tightened to the equality $\rank(P) = \rank(L)\rank(\Ph)$ (a fact that follows easily once we have derived the eigenvalues of $P$).  But for now, the inequality suffices and we proceed to analyze the rank of $\Ph$. To that end, Lemma~\ref{lem.rank} proves useful.

\begin{lemma}
  \label{lem.rank}
  Let $A=UU^T$, where $U$ is any $n \times 2r$ ($2r \le n$) matrix of rank $2r$. Let $J_{2r}$ be the `symplectic identity'
  \begin{equation*}
    J_{2r} :=
    \begin{bmatrix}
      0 & I_r\\
      -I_r & 0
    \end{bmatrix},
  \end{equation*}
  where $I_r$ is the $r \times r$ identity matrix. Let $B=UJ_{2r}U^T$. Then, the rank of the block symmetric matrix
  \begin{equation}
    \label{eq.33}
    Z =  \begin{bmatrix}
      A & B\\
      B^T & A
    \end{bmatrix},
  \end{equation}
  equals the rank of $A$, i.e., $\rank(Z) = \rank(A)$.
\end{lemma}
\begin{proof}
  First, recall two basic facts (see e.g.,~\cite{horn85}) about ranks of matrix products:
  \begin{enumerate}[(i)]
  \item Let $X$ be $m \times n$ and $Y$ any $n \times p$ matrix with rank $n$. Then,  $\rank(XY)=\rank(X)$; and
  \item Let $X$ be $m \times n$ and $Y$ any $p \times m$ matrix with rank $m$. Then,  $\rank(YX)=\rank(X)$.
  \end{enumerate}
  It is easy to see that $Z$ may be factorized as
  \begin{equation*}
    Z =
    \begin{bmatrix}
      UU^T & UJ_rU^T\\
      UJ_r^TU^T & UU^T
    \end{bmatrix}\quad=\quad
    \begin{bmatrix}
      U & 0\\
      0 & U
    \end{bmatrix}
    L_{2r}
    \begin{bmatrix}
      U^T & 0\\
      0 & U^T
    \end{bmatrix},\quad\text{where}\ L_{2r} := 
    \begin{bmatrix}
      I_{2r} & J_{2r}\\
      -J_{2r} & I_{2r}
    \end{bmatrix}.
  \end{equation*}
  Since $U$ has rank $2r$, the rank of the direct sum $\rank(U\oplus I_2) = 4r$; similarly $\rank(U^T\oplus I_2) = 4r$. Thus, using Properties~(i) and~(ii) of ranks, we conclude that $\rank(Z)=\rank(L_{2r})$. But Since $I_{2r}$ and $J_{2r}$ are invertible, elementary manipulations show that $\rank(L_{2r}) = \rank(I_{2r}) + \rank(I_{2r} - J_{2r}I_{2r}(-J_{2r})) = \rank(I_{2r}) = 2r$, which follows upon noting $J_{2r}J_{2r}=-I_{2r}$.\qedhere
\end{proof}
\noindent As a consequence of~\ref{lem.rank} we immediately have the following corollary.
\begin{corr}
  \label{corr.rank}
  Let $\Ph$ be as defined by~(\ref{eq.27}). Then, $\rank(\Ph) = 2$.
\end{corr}
\begin{proof}
  Notice that $\ah_{ij}=\cos(x_i-x_j) = \cos(x_i)\cos(x_j)+\sin(x_i)\sin(x_j)$, while $\bh_{ij}=\sin(x_i-x_j)=\sin(x_i)\cos(x_j)-\cos(x_i)\sin(x_j)$. Thus, $\Ah$ is of the form $UU^T$, where $\rank(U)=2$, while $\Bh$ is of the form $UJ_2U^T$, for a suitable $n \times 2$ real matrix $U$. Thus, $\rank(\Ph) = \rank(\Ah) = 2$.
\end{proof}

\section{Explicit eigenvalues}
\label{sec.eigvals}
In this section we derive explicit expressions for eigenvalues of $P$. Here Corollary~\ref{corr.rank} plays a key role since it allows us to conclude that $\rank(P) \le 4$ (using~(\ref{eq.31})). Thereafter, basic linear algebra shows that the nonzero eigenvalues of $P$ lie among the roots of its reduced order-4 characteristic polynomial:
\begin{equation}
  \label{eq.7}
  \chi_P(\lambda) := \lambda^4 - \phi_1\lambda^3 + \phi_2\lambda^2 - \phi_3\lambda + \phi_4,
\end{equation}
where $\phi_m$ is the $m$th elementary symmetric polynomial of $P$. Using~\eqref{eq.7} we can now proceed onto computing the eigenvalues of $P$. We will see that when $\rank(L)=2$, the characteristic polynomial~\eqref{eq.7} actually has $4$ nonzero roots, thus actually the equality $\rank(P)=4$ holds.

All that remains is to factorize the polynomial~\eqref{eq.7}, which will then yield the desired eigenvalues. To that end, we first compute the coefficients $\phi_m$, a job that is simplified by invoking the well-known Newton's identities for symmetric functions~\citep[see e.g.,][]{mead92}, and the basic fact that $\sum_i \lambda_i^m(P) = \trace P^m$. Specifically, we have the following identities:
\begin{equation}
  \label{eq.8}
  \begin{split}
    \phi_1 &= \trace P\\
    2\phi_2 &= (\trace P)^2- \trace P^2\\
    3\phi_3 &= \phi_2\trace P - \trace P\trace P^2 + \trace P^3\\
    4\phi_4 &= \phi_3\trace P - \phi_2\trace P^2 + \trace P\trace P^3 - \trace P^4.
  \end{split}
\end{equation}
To compute $\phi_m$ ($1\le m\le 4$) via~\eqref{eq.8} we require knowledge of $\trace(P^m)$. And it is at this point where we may hope for benign simplification due to the special trigonometric structure of $P$. Here are the details.

\begin{lemma}[Powers]
  \label{lem.powers}
  Let $P$ and $L$ be as defined by equations~(\ref{eq.26}) and~\eqref{eq.25}, respectively. Then,
  \begin{equation}
    \label{eq.3}
    \trace( P^m) = 2\trace(L^m),\quad m=1, 2, 3, 4.
  \end{equation}
\end{lemma}

\begin{proof}
  The case $m=1$ is obvious, since $a_{ii}=1$. For $m > 1$, we propose computing the following formulae:
 \begin{align}
   \trace( P^2) &= \nlsum_{i,j=1}^{2n} p_{ij}p_{ij}\\
   \trace( P^3) &= \nlsum_{i,j=1}^{2n} p_{ij}^2p_{ij}\\
   \trace( P^4) &= \nlsum_{i,j=1}^{2n} p_{ij}^2p_{ij}^2.
\end{align}
We split each of traces above into 4 parts, corresponding to the 4 blocks of $P$. Pictorially,
\begin{center}
  $P\quad${\Large$\rightsquigarrow$}$\quad$  \begin{tabular}{|c|c|}
    \hline
    $S_1$ & $S_2$\\
    \hline
    $S_3$ & $S_4$\\
    \hline
  \end{tabular}\ \ ,
\end{center}
where $S_i$ (\fromto{1}{i}{4}) denotes the contribution to the trace from the corresponding block of $P^m$.

\noindent{\it (1) The case $m=2$.}
  \begin{align*}
    \trace(P^2) &= \nlsum_{i,j=1}^{2n} p_{ij}p_{ij}\\
    &= \sum_{i,j=1}^np_{ij}^2 + \sum_{i=1}^n\sum_{j=n+1}^{2n}p_{ij}^2 + \sum_{i=n+1}^{2n}\sum_{j=1}^np_{ij}^2 + \sum_{i=n+1}^{2n}\sum_{j=n+1}^{2n}p_{ij}^2\\
    &= S_1 + S_2 + S_3 + S_4.
  \end{align*}
  We compute each of the sums $S_i$ now. The algebraic simplifications exploit two main points: (i) symmetry of $L$; and (ii) the fact that $\sin(x)$ is an odd function. 
  \begin{align*}
    S_1 &= \sum_{i,j=1}^np_{ij}^2 = \sum_{i,j=1}^na_{ij}^2= \sum_{i,j=1}^n\ah_{ij}^2l_{ij}^2\\
    S_2  &= \sum_{i=1}^n\sum_{j=n+1}^{2n}p_{ij}^2 = \sum_{i=1}^n\sum_{j=1}^nb_{ij}^2  = \sum_{i=1}^n\sum_{j=1}^n \bh_{ij}^2l_{ij}^2\\
    S_3 &= \sum_{i=n+1}^{2n}\sum_{j=1}^np_{ij}^2 = \sum_{i=1}^n\sum_{j=1}^nb_{ji}^2=\sum_{i,j=1}^n\bh_{ji}^2l_{ij}^2\\
    S_4 &= \sum_{i=n+1}^{2n}\sum_{j=n+1}^{2n} p_{ij}^2 = \sum_{i,j=1}^na_{ij}^2= S_1.
  \end{align*}
  Notice, although $b_{ij}=-b_{ji}$, their squares are the same. Thus, $S_2=S_3$, whereby we obtain
  \begin{equation*}
    \trace(P^2) = 2(S_1+S_3) = 2\nlsum_{ij}(\ah_{ij}^2+\bh_{ij}^2)l_{ij}^2 = 2\nlsum_{ij}(\cos^2(x_i-x_j)+\sin^2(x_i-x_j))l_{ij}^2 = 2\trace(L^2).
  \end{equation*}

\noindent{\it (2) The case $m=3$.}

\noindent For clarity, we introduce the shorthand $\theta_{ij}=x_i-x_j$. 
  \begin{equation*}
    \trace(P^3) = \nlsum_{i,j=1}^{2n} [P^2]_{ij}p_{ij} = S_1+S_2+S_3 + S_4.
  \end{equation*}
  Direct multiplication shows that $P^2$ is given by
  \begin{equation}
    \label{eq.4}
    P^2\quad=\quad
    \begin{bmatrix}
      A & B\\
      B^T & A
    \end{bmatrix} \begin{bmatrix}
      A & B\\
      B^T & A
    \end{bmatrix}\quad=\quad
    \begin{bmatrix}
      A^2 + BB^T & AB+BA\\
      B^TA + AB^T & B^TB+A^2
    \end{bmatrix}.
  \end{equation}
  Using~\eqref{eq.4}, we now compute each of the partial sums $S_1$ through $S_4$.
  \begin{align*}
    S_1 &= \sum_{i,j=1}^n( [A^2]_{ij} + [BB^T]_{ij})\ah_{ij}l_{ij}= \nlsum_{ijk}\ah_{ij}l_{ij}(a_{ik}a_{kj} + b_{ik}b_{jk})\\
    &= \nlsum_{ijk}\ah_{ij}l_{ij}l_{ik}l_{kj}(\ah_{ik}\ah_{kj}+ \bh_{ik}\bh_{jk}) = \nlsum_{ijk}\ah_{ij}l_{ij}l_{ik}l_{kj}(\cos\theta_{ik}\cos\theta_{kj} - \sin\theta_{ik}\sin\theta_{kj})\\
    &= \nlsum_{ijk}\ah_{ij}l_{ij}l_{ik}l_{kj}(\cos(\theta_{ik}+\theta_{kj})) = \nlsum_{ijk}\ah_{ij}^2l_{ij}l_{ik}l_{kj} = \nlsum_{ij}\ah_{ij}^2l_{ij}[L^2]_{ij}.
  \end{align*}
  \begin{align*}
    S_2 &= \nlsum_{ij}[AB+BA]_{ij}b_{ij} = \nlsum_{ijk}\bh_{ij}l_{ij}(\ah_{ik}l_{ik}l_{kj}\bh_{kj} + \bh_{ik}l_{ik}l_{kj}\ah_{kj})\\
    &= \nlsum_{ijk}\bh_{ij}l_{ij}l_{ik}l_{kj}(\ah_{ik}\bh_{kj} + \bh_{ik}\ah_{kj}) = 
    \nlsum_{ijk}\bh_{ij}l_{ij}l_{ik}l_{kj}(\cos\theta_{ik}\sin\theta_{kj} + \sin\theta_{ik}\cos\theta_{kj})\\
    &=\nlsum_{ijk}\bh_{ij}l_{ij}l_{ik}l_{kj}\sin(\theta_{ik}+\theta_{kj}) = \nlsum_{ijk}\bh_{ij}^2l_{ij}l_{ik}l_{kj} =
    \nlsum_{ij}\bh_{ij}^2l_{ij}[L^2]_{ij}.
  \end{align*}
  Observe that the derivation below depends on $\bh_{ij}=-\bh_{ji}$, $\bh_{ki}=-\bh_{ik}$, and $\bh_{kj}=-\bh_{jk}$.
  \begin{align*}
    S_3 &= \nlsum_{ij}[B^TA+AB^T]_{ij}b_{ji} = \nlsum_{ij}\bh_{ji}l_{ij}\nlsum_k (b_{ki}a_{kj} + a_{ik}b_{jk})\\
    &= \nlsum_{ij}\bh_{ji}l_{ij}l_{ik}l_{kj}(\bh_{ki}\ah_{kj} + \ah_{ik}\bh_{jk}) = \nlsum_{ij}\bh_{ij}l_{ij}l_{ik}l_{kj}(\bh_{ik}\ah_{kj}-\ah_{ik}\bh_{kj})\\
    &=\nlsum_{ijk}\bh_{ji}l_{ij}l_{ik}l_{kj}(\sin\theta_{ik}\cos\theta_{kj}+\cos\theta_{ik}\sin\theta_{kj}) = \nlsum_{ijk}\bh_{ij}l_{ij}l_{ik}l_{kj} \sin(\theta_{ik}+\theta_{kj})\\
    &= \nlsum_{ijk}\bh_{ij}^2l_{ij}l_{ij}l_{ik}l_{kj} = \nlsum_{ij}\bh_{ij}^2l_{ij}[L^2]_{ij} = S_2.
  \end{align*}
  Similar manipulations show that $S_4=S_1$. Thus, we have
  \begin{equation}
    \label{eq.1}
    \trace(P^3) = 2(S_1+S_3) = \nlsum_{ij}(\ah_{ij}^2+\bh_{ij}^2)l_{ij}[L^2]_{ij} = 2\trace(L^3).
  \end{equation}

\noindent\underline{\it The case $m=4$.}

\noindent  Given the derivations above, one may safely guess that
  \begin{equation*}
    \trace(P^4) = 2\trace(L^4).
  \end{equation*}
  Let us explicitly see why this guess is true. As before, we write
  \begin{equation*}
    \trace(P^4) = \nlsum_{ij} [P^2]_{ij}[P^2]_{ij} = S_1 + S_2 + S_3 + S_4.
  \end{equation*}
  \begin{align*}
    S_1 &= \nlsum_{ij}([A^2]_{ij}+[BB^T]_{ij})^2 = \nlsum_{ij} \left( \nlsum_k a_{ik}a_{kj}+b_{ik}b_{jk}\right)^2\\
    &= \nlsum_{ij}\left(\nlsum_k l_{ik}l_{kj}(\ah_{ik}\ah_{kj}-\bh_{ik}\bh_{kj}) \right)^2\\
    &= \nlsum_{ij}\left(\nlsum_k l_{ik}l_{kj}(\cos\theta_{ik}\cos\theta_{kj}-\sin\theta_{ik}\sin\theta_{kj})\right)^2 = 
    \nlsum_{ij}\left(\nlsum_k l_{ik}l_{kj}\cos(\theta_{ik}+\theta_{kj}) \right)^2\\
    &= \nlsum_{ij}\left(\nlsum_k l_{ik}l_{kj}\ah_{ij}\right)^2 = \nlsum_{ij}\ah_{ij}^2[L^2]_{ij}^2.
  \end{align*}
  Similarly, by now we routinely see that
  \begin{align*}
    S_2 &= \nlsum_{ij}\left([AB]_{ij}+[BA]_{ij}\right)^2 = \nlsum_{ij}\bh_{ij}^2[L^2]_{ij}^2,
  \end{align*}
  and also that $S_3=S_2$ and $S_4=S_1$. Therefore, we finally have
  \begin{equation*}
    \trace(P^4) = 2(S_1+S_2) = 2\nlsum_{ij}(\ah_{ij}^2+\bh_{ij}^2)[L^2]_{ij}^2 = 2\nlsum_{ij}[L^2]_{ij}^2 = 2\trace(L^2L^2)=2\trace(L^4).\qedhere
  \end{equation*}
\end{proof}

Using Lemma~\ref{lem.powers} we obtain $\phi_m$ as shown by Proposition~\ref{prop.roots} below.
\begin{prop}
  \label{prop.roots}
  If $\rank(L)=2$, then the nonzero eigenvalues of the generalized trigonometric matrix $P$ given by~(\ref{eq.26}) are
\begin{equation}
  \label{eq.35}
  \lambda_{1,2} = \gamma + \delta,\quad \lambda_{3,4}=\gamma-\delta,
\end{equation}
and the rest of the eigenvalues are zero.
\end{prop}
\begin{proof}
  First, we introduce the quantities
  \begin{equation}
    \label{eq.16}
    \gamma_m := 2\trace(L^m),\quad m=1,2,3,4,
  \end{equation}
  and then plug them into~(\ref{eq.8}) to obtain
  \begin{equation}
    \label{eq.15}
    \phi_1 = \gamma_1;\quad\phi_2 = \half(\gamma_1^2- \gamma_2);\quad \phi_3 = \tfrac{1}{6}\gamma_1^3-\tfrac{1}{2}\gamma_2\gamma_1 + \tfrac{1}{3}\gamma_3;\quad 4\phi_4 = \phi_3\gamma_1 - \phi_2\gamma_2 + \gamma_1\gamma_3 - \gamma_4.
  \end{equation}
  We must now compute roots of the quartic polynomial $\chi_P(\lambda)$ given by (\ref{eq.7}). To obtain its roots explicitly, let us further refine the values that $\phi_m$ can take, by computing $\gamma_m=\sum_i \lambda_i^m(L)$ explicitly. 

  Thus, we first compute the eigenvalues of $L$. Recall that $L = hl^T + lh^T$ is at most a rank-2 matrix. Thus, its nonzero eigenvalues must lie among the roots of the reduced polynomial
  \begin{equation}
    \label{eq.14}
    \chi_L(\lambda) := \lambda^2 - \phi_1(L)\lambda + \phi_2(L) = 0,
  \end{equation}
  where $\phi_i(L)$ ($i=1$, $2$) are elementary symmetric polynomials of matrix $L$. Now define the quantities
  \begin{equation}
    \label{eq.23}
    \gamma = l^Th,\quad\quad \delta = [(l^Tl)(h^Th)]^{1/2}.
  \end{equation}
  Then, $\phi_1(L)=\trace L = 2\gamma$, while $\phi_2(L)=\half((\trace L)^2-\trace L^2) = \gamma^2-\delta^2$, since
  \begin{equation*}
    \trace L^2 = \trace (lh^T+hl^T)(lh^T+hl^T) = \trace\bigl(lh^Tlh+lh^Thl^T+hl^Tlh^T+hl^Thl^T\bigr)=2(\gamma^2+\delta^2).
  \end{equation*}
  Plugging in these values into~\eqref{eq.14} we immediately see that the quadratic factorizes as
  \begin{equation}
    \label{eq.21}
    \chi_L(\lambda) = \lambda^2 - 2\gamma\lambda + \gamma^2-\delta^2 = (\gamma -\delta -\lambda ) (\gamma +\delta -\lambda).
  \end{equation}
  Thus, the (nonzero, unless $\gamma=\delta$) eigenvalues of $L$ are
  \begin{equation}
    \label{eq.22}
    \lambda_1 = \gamma + \delta,\qquad \lambda_2 = \gamma - \delta.
  \end{equation}
  Using~\eqref{eq.22} and performing some algebra we obtain the following equations:
  \begin{subequations}
    \begin{align}
      \label{eq.10}
      \gamma_1\quad&=\quad 2(\lambda_1+\lambda_2) = 4\gamma\\
      \label{eq.11}
      \gamma_2\quad&=\quad 2(\lambda_1^2+\lambda_2^2) = 4(\gamma^2 + \delta^2)\\
      \label{eq.12}
      \gamma_3\quad&=\quad 2(\lambda_1^3+\lambda_2^3) = 4\left(\gamma^3 + 3\gamma\delta^2\right)\\
      \label{eq.13}
      \gamma_4\quad&=\quad 2(\lambda_1^4+\lambda_2^4) = 4\left(\gamma^4+6 \gamma^2\delta^2 +\delta^4\right).
    \end{align}
  \end{subequations}
  Using~\eqref{eq.10}--\eqref{eq.13} and the definitions~\eqref{eq.15}, we obtain $\phi_m$ as follows
  \begin{subequations}
    \begin{alignat}{3}
      \label{eq.32}
      &\phi_1\quad=\quad 4\gamma,\qquad&\phi_2&\quad=\quad 6\gamma^2 - 2\delta^2\\
      \label{eq.34}
      &\phi_3\quad=\quad 4(\gamma^3 - \gamma\delta^2),\qquad&\phi_4&\quad=\quad (\gamma^2-\delta^2)^2.
    \end{alignat}
  \end{subequations}
  Recall now the reduced characteristic polynomial for $P$
  \begin{equation}
    \label{eq.28}
    \chi_P(\lambda) = \lambda^4 - \phi_1\lambda^3 + \phi_2\lambda^2 - \phi_3\lambda + \phi_4,
  \end{equation}
  into which we substitute~\eqref{eq.32}--\eqref{eq.34} and simplify to obtain the factorization
  \begin{equation}
    \label{eq.17}
    \chi_P(\lambda) = \bigl(\lambda-\gamma +\delta\bigr)^2 \bigl(\lambda-\gamma -\delta\bigr)^2.
  \end{equation}
  This immediately yields the desired roots
  \begin{equation*}
    \lambda_{1,2} = \gamma + \delta,\quad \lambda_{3,4}=\gamma-\delta.\qedhere
  \end{equation*}
\emph{Remark:} Observe that $\gamma=\delta$ can hold if only if $\rank(L) = 1$ (we ignore the trivial case of rank-0 $L$). In this case $P$ has only two nonzero eigenvalues.
\end{proof}
As a corollary we immediately obtain a proof of Conjecture~\ref{conj.one}, which was posed by~\citet{zhLiLi12}.
\begin{corr}
  \label{corr.conj}
  Let the matrix $P(\omega)$ be as defined by~\eqref{eq.71} and~\eqref{eq.18}. Then, it has 2 positive and 2 negative eigenvalues given by
  \begin{equation}
    \label{eq.19}
    \lambda_+(P) = \frac{n}{4}\left(n-1 + \sqrt{\frac{4n^2-6n+2}{3}}\right),\quad\quad \lambda_-(P) = \frac{n}{4}\left(n-1 - \sqrt{\frac{4n^2-6n+2}{3}}\right).
  \end{equation}
\end{corr}
\begin{proof}
  From~\eqref{eq.18} we see that $L = le^T+el^T$, so that $\delta=\sqrt{nl^Tl}=\sqrt{n}\norm{l}$. Moreover, $l_i = (i-1)/2$ for $1 \le i \le n$. Thus, the eigenvalues of $P$ are given by the roots of~\eqref{eq.17}, which equal
  \begin{equation}
    \label{eq.20}
    \lambda_+ = \gamma + \sqrt{n}\norm{l},\qquad\lambda_- = \gamma - \sqrt{n}\norm{l},
  \end{equation}
  both with multiplicity 2. Since $l_i = (i-1)/2$, we have
  \begin{equation*}
      \gamma = l^Te = \tfrac{1}{4}n(n-1),\qquad \norm{l} = \sqrt{\tfrac{1}{24}(2n^3-3n^2+n)} = \frac{\sqrt{n}}{4}\sqrt{\frac{4n^2 - 6n + 2}{3}}.
\end{equation*}
Plugging these values of $\gamma$ and $\norm{l}$ into~\eqref{eq.20} and simplifying, we obtain~\eqref{eq.19}.
\end{proof}

We conclude by mentioning that instead of trigonometric matrices, we can carry out a similar derivation for more general rank-2 matrices that satisfy the hypotheses of Lemma~\ref{lem.rank}. But the details are laborious, so we leave them as an exercise for the interested reader.

\bibliographystyle{plainnat}

\end{document}